\documentclass[11pt]{amsart}
\usepackage{amssymb,mathtools,amsthm}
\usepackage{subsupscripts}

\numberwithin{equation}{section}

\theoremstyle{plain}
\newtheorem{thm}{Theorem}[section]
\newtheorem{prop}[thm]{Proposition}

\newtheorem{cor}[thm]{Corollary}

\theoremstyle{definition}

\newtheorem{rem}[thm]{Remark}

\newcommand{\Z}{\mathbb{Z}}

\newcommand{\Li}{\mathrm{Li}}
\newcommand{\bk}{{\boldsymbol{k}}}
\newcommand{\ba}{{\boldsymbol{a}}}
\newcommand{\bb}{{\boldsymbol{b}}}

\newcommand{\varray}[1]{\!\!\begin{array}{c}#1\end{array}\!}
\newcommand{\trF}[3]{{}_{#1}^{\phantom{(}}F_{#2}^{#3}}

\renewcommand{\Re}{\operatorname{Re}}

\title{Truncated Hypergeometric Functions and Discretized Integrals}
\author{Shuji Yamamoto}

\thanks{This research was supported by JSPS KAKENHI JP21K03185 
and JST, CREST Grant Number JPMJCR1913, Japan. }

\address{Interdisciplinary Faculty of Science and Engineering, Shimane University, 
1060 Nishi-Kawatsu, Matsue, 690-8504, Japan. }
\email{yamashu@riko.shimane-u.ac.jp}

\begin{document}
\begin{abstract}
We introduce a kind of finite truncation of the hypergeometric series 
and provide its discretized integral representation. 
This is motivated by recent results of Maesaka-Seki-Watanabe and Hirose-Matsusaka-Seki 
on the identity between truncated series and discretized integrals 
arising from multiple zeta values and multiple polylogarithms. 
We also prove the formula of Ohno-Zagier type which relates the truncated multiple polylogarithms 
and the truncated hypergeometric series. 
\end{abstract}

\maketitle

\section{Introduction}
In 2024, Maesaka, Seki and Watanabe \cite{MSW} discovered an interesting identity 
\begin{equation}\label{eq:MSW}
\sum_{0<m_1<\cdots<m_r<N}\frac{1}{m_1^{k_1}\cdots m_r^{k_r}}
=\sum_{\substack{0<n_{i,1}\le\cdots\le n_{i,k_i}<N\ (1\le i\le r)\\ n_{i,k_i}<n_{i+1,1}\ (1\le i<r)}}
\prod_{i=1}^r\frac{1}{(N-n_{i,1})n_{i,2}\cdots n_{i,k_i}}, 
\end{equation}
where $k_1,\ldots,k_r$ and $N$ are positive integers. 
When $k_r\ge 2$, the left-hand side of \eqref{eq:MSW} is the \emph{truncation} of the series 
\[\zeta(k_1,\ldots,k_r)=\sum_{0<m_1<\cdots<m_r}\frac{1}{m_1^{k_1}\cdots m_r^{k_r}}, \]
called the \emph{multiple zeta value}. As is well known, $\zeta(k_1,\ldots,k_r)$ also has 
the iterated integral representation 
\[\zeta(k_1,\ldots,k_r)
=\int_{\substack{0<t_{i,1}<\cdots<t_{i,k_i}<1\ (1\le i\le r)\\ t_{i,k_i}<t_{i+1,1}\ (1\le i<r)}}
\prod_{i=1}^r\frac{dt_{i,1}dt_{i,2}\cdots dt_{i,k_i}}{(1-t_{i,1})t_{i,2}\cdots t_{i,k_i}}, \]
and the right-hand side of \eqref{eq:MSW} is its \emph{discretization}, that is, a Riemann sum 
approximating the integral. 
Thus, at least when $k_r\ge 2$, the identity \eqref{eq:MSW} can be regarded as a refinement 
of the identity between the series and integral representations of the multiple zeta value. 

As a generalization of \eqref{eq:MSW}, 
Hirose, Matsusaka and Seki \cite{HMS} proved the formula 
\begin{equation}\label{eq:HMS}
\begin{split}
\sum_{0<m_1<\cdots<m_r<N}&\frac{1}{m_1^{k_1}\cdots m_r^{k_r}}
\prod_{i=1}^r\frac{(Nx_{i+1}-m_i)_{m_i}}{(Nx_{i}-m_i)_{m_i}}\\
&=\sum_{\substack{0<n_{i,1}\le\cdots\le n_{i,k_i}<N\ (1\le i\le r)\\ n_{i,k_i}<n_{i+1,1}\ (1\le i<r)}}
\prod_{i=1}^r\frac{1}{(Nx_i-n_{i,1})n_{i,2}\cdots n_{i,k_i}},  
\end{split}
\end{equation}
where $x_1,\ldots,x_r$ are indeterminates, $x_{r+1}=1$ and 
$(a)_m=a(a+1)\cdots(a+m-1)$ denotes the rising factorial. 
This is an analogue of the identity between the series and integral representations of 
the multiple polylogarithm 
\begin{equation}\label{eq:MPL int}
\begin{split}
\sum_{0<m_1<\cdots<m_r}&\frac{1}{m_1^{k_1}\cdots m_r^{k_r}}
\prod_{i=1}^r\biggl(\frac{x_{i+1}}{x_i}\biggr)^{m_i}\\
&=\int_{\substack{0<t_{i,1}<\cdots<t_{i,k_i}<1\ (1\le i\le r)\\ t_{i,k_i}<t_{i+1,1}\ (1\le i<r)}} 
\prod_{i=1}^r\frac{dt_{i,1}dt_{i,2}\cdots dt_{i,k_i}}{(x_i-t_{i,1})t_{i,2}\cdots t_{i,k_i}}. 
\end{split}
\end{equation}
Although the left-hand side of \eqref{eq:HMS} is not a simple truncation of 
the series in \eqref{eq:MPL int} and instead involves the modification 
\begin{equation}\label{eq:modification}
\biggl(\frac{x_{i+1}}{x_i}\biggr)^{m_i}\rightsquigarrow 
\frac{(Nx_{i+1}-m_i)_{m_i}}{(Nx_{i}-m_i)_{m_i}}
=\frac{(Nx_{i+1}-1)\cdots(Nx_{i+1}-m_i)}{(Nx_{i}-1)\cdots(Nx_{i}-m_i)}.  
\end{equation}
For complex values of $x_1,\ldots,x_r$ satisfying $|x_i|\ge 1$ and $(k_r,x_r)\ne(1,1)$, 
Hirose-Matsusaka-Seki showed that the left-hand side of \eqref{eq:HMS} 
tends to that of \eqref{eq:MPL int} (see \cite[Proposition 2.5]{HMS}).  

More generally, let us consider the following question: 
Given an identity of the form ``series = integral'', 
does it have a finite analogue of the form 
\begin{equation}\label{eq:trunc disc}
\text{``truncated series = discretized integral''}? 
\end{equation}
We do not define precisely the terms ``truncated series'' and ``discretized integral'', 
but roughly speaking, we use them in the following sense:  
\begin{itemize}
\item $\sum_{m=0}^N a^{(N)}_m$ is a \emph{truncation} of $\sum_{m=0}^\infty a_m$ 
if $a^{(N)}_m\to a_m$ as $N\to\infty$, for each $m$. 
\item $\frac{1}{N}\sum_{n=0}^N f^{(N)}(n)$ is a \emph{discretization} of $\int_0^1 f(t)dt$ 
if $f^{(N)}(n)\to f(t)$ as $n,N\to\infty$ with $n/N\to t$, for each $0<t<1$. 
\end{itemize}
Thus we admit modifications like \eqref{eq:modification} in both. 
One may also require the convergences 
\[\sum_{m=0}^N a^{(N)}_m\to\sum_{m=0}^\infty a_m \quad\text{ and }\quad 
\frac{1}{N}\sum_{n=0}^N f^{(N)}(n)\to\int_0^1 f(t)dt, \]
but we will not do so in this paper. 

The purpose of the present paper is to give a finite analogue, in the above sense, 
of the Euler integral formula 
\begin{equation}\label{eq:HG int}
{}_{2}F_{1}\biggl(\varray{a,b\\ c};z\biggr)
=\frac{\Gamma(c)}{\Gamma(a)\Gamma(c-a)}\int_0^1 t^{a-1}(1-t)^{c-a-1}(1-tz)^{-b}dt
\end{equation}
for the hypergeometric series 
\begin{equation}\label{eq:HG}
{}_2F_1\biggl(\varray{a,b\\ c};z\biggr)
=\sum_{m=0}^\infty\frac{(a)_m(b)_m}{(c)_m m!}z^m, 
\end{equation}
where $|z|<1$ and $\Re c>\Re a>0$. 
We define the truncated hypergeometric function by 
\begin{equation}\label{eq:tHG}
\trF{2}{1}{[N]}\biggl(\varray{a,b\\ c};z\biggr)
\coloneqq \sum_{m=0}^{N}\frac{(a)_m(b)_m}{(c)_m m!}\frac{(N+1-m)_m}{(Nz^{-1}-m)_m}
\end{equation}
for any integer $N\ge 0$. 
Then we have the following finite analogue of \eqref{eq:HG int}: 

\begin{thm}[= Theorem \ref{thm:tHG int}]\label{thm:tHG int intro}
\[
\trF{2}{1}{[N]}\biggl(\varray{a,b\\ c};z\biggr)
=\frac{(a)_{N} (c-a)_{N}}{(c)_{N} N!}
\sum_{n=0}^{N}\frac{(1+n)_{N-n}}{(a+n)_{N-n}}\frac{(1+N-n)_{n}}{(c-a+N-n)_{n}}
\frac{(b+Nz^{-1}-n)_{n}}{(Nz^{-1}-n)_{n}}. 
\]
\end{thm}

Here the factor $(a)_{N} (c-a)_{N}/(c)_{N}(N-1)!$ is the finite version of $\Gamma(c)/\Gamma(a)\Gamma(c-a)$; 
in fact, the former is the reciprocal of the \emph{truncated beta function} $B^{[N]}(a,c-a)$, 
which will be introduced in \S2. 

\begin{rem}
Theorem \ref{thm:tHG int intro} is understood as an identity of 
rational functions in $a,b,c,z$, so no convergence issues arise. 
The finite analogues considered below are interpreted similarly. 
For the classical formulas quoted for comparison, we generally 
leave the standard convergence conditions implicit.
\end{rem}

The truncation $\trF{2}{1}{[N]}(z)$ is essentially equivalent to the terminating ${}_3F_2(1)$: 
\begin{equation}\label{eq:terminating 3F2}
\trF{2}{1}{[N]}\biggl(\varray{a,b\\ c};z\biggr)={}_3F_2\biggl(\varray{a,b,-N\\ c,1-Nz^{-1}};1\biggr). 
\end{equation}
This relation enables us to derive certain formulas on the former, 
including Theorem \ref{thm:tHG int intro}, from (often previously known) results on the latter. 
To put it another way, we can interpret some results on the terminating ${}_3F_2(1)$ 
as finite analogues of formulas on ${}_2F_1$. 

Compared with Hirose-Matsusaka-Seki's truncated multiple polylogarithms \eqref{eq:HMS}, 
it may look more natural to define the truncated hypergeometric function as 
\[\trF{2}{1}{(N)}\biggl(\varray{a,b\\ c};z\biggr)
\coloneqq \sum_{m=0}^{N-1}\frac{(a)_m(b)_m}{(c)_m m!}\frac{(N-m)_m}{(Nz^{-1}-m)_m}. \]
The two definitions are related by 
\[\trF{2}{1}{(N)}\biggl(\varray{a,b\\ c};z\biggr)
=\trF{2}{1}{[N-1]}\biggl(\varray{a,b\\ c};\frac{N-1}{N}z\biggr), \]
hence these two types of truncation are essentially equivalent. 
We mainly adopt $\trF{2}{1}{[N]}$ because the formulas are often cleaner than those for $\trF{2}{1}{(N)}$. 
An exception is the Ohno-Zagier type formula, which provides a relationship 
between the truncated hypergeometric function and the truncated multiple polylogarithms: 

\begin{thm}[= Theorem \ref{thm:tOZ}]\label{thm:tOZ intro}
We have 
\[
\Phi_0^{(N)}(X,Y,Z;z)=\frac{1}{Z-XY}
\biggl\{\trF{2}{1}{(N)}\biggl(\varray{\alpha-X,\beta-X\\ 1-X};z\biggr)-1\biggr\}, 
\]
where $\Phi_0^{(N)}$ denotes the generating function of sums of truncated multiple polylogarithms 
of fixed weight, depth and height (see \eqref{eq:Phi_0^(N) def} for the precise definition), 
and $\alpha$ and $\beta$ satisfy the relations 
\[\alpha+\beta=X+Y,\quad \alpha\beta=Z. \]
\end{thm}

The contents of this paper are as follows. 
In Section 2, we first define the truncated beta function by truncating 
the infinite \emph{product} representation of the beta function. 
Then we give a discretized integral representation for it. 
In Section 3, we introduce the truncated generalized hypergeometric function $\trF{p}{q}{[N]}$. 
After computing $\trF{1}{0}{[N]}$, we prove the discretized integral formula 
representing $\trF{p+1}{q+1}{[N]}$ in terms of $\trF{p}{q}{[n]}$ ($0\le n\le N$). 
In particular, we obtain Theorem \ref{thm:tHG int intro}. 
In Section 4, we present two applications of Theorem \ref{thm:tHG int intro}, 
namely, the finite analogues of Gauss' hypergeometric theorem and 
Pfaff's and Euler's transformation formulas. 
In Section 5, we study the Ohno-Zagier type formula (Theorem \ref{thm:tOZ intro}) 
and its consequences. 
In Section 6, we treat truncation of certain generalized hypergeometric series of two variables. 
In particular, we prove the finite analogue of the integral representation 
for Appell's hypergeometric function $F_3$. 

\section{Truncated beta function}
The product expression of the gamma function 
\[\Gamma(a)=a^{-1}e^{-\gamma a}\prod_{n=1}^\infty\biggl(1+\frac{a}{n}\biggr)^{-1}e^{a/n} \]
implies 
\[B(a,b)=\frac{\Gamma(a)\Gamma(b)}{\Gamma(a+b)}
=\frac{a+b}{ab}\prod_{n=1}^\infty\frac{(a+b+n)n}{(a+n)(b+n)}. \]
For a positive integer $N$, we define the \emph{truncated beta function} as follows: 
\[B^{[N]}(a,b)
\coloneqq\frac{a+b}{ab}\prod_{n=1}^{N-1}\frac{(a+b+n)n}{(a+n)(b+n)}
=\frac{(a+b)_N\,(N-1)!}{(a)_N (b)_N}. \] 

\begin{prop}\label{prop:disc beta}
For any nonnegative integer, we have 
\[\frac{(a+b)_N N!}{(a)_N (b)_N}
=\sum_{n=0}^N \frac{(1+n)_{N-n}}{(a+n)_{N-n}}\frac{(1+N-n)_n}{(b+N-n)_n}. \]
Hence, for $N>0$, we have 
\begin{equation}\label{eq:disc beta}
B^{[N]}(a,b)=
\frac{1}{N}\sum_{n=0}^N \frac{(1+n)_{N-n}}{(a+n)_{N-n}}\frac{(1+N-n)_n}{(b+N-n)_n}. 
\end{equation}
\end{prop}

\begin{proof}
First recall the identity 
\begin{equation}\label{eq:binomial}
\sum_{n=0}^N\frac{(a)_n}{n!}\frac{(b)_{N-n}}{(N-n)!}=\frac{(a+b)_N}{N!}, 
\end{equation}
which follows from the identity of generating series 
\[(1-x)^{-a}(1-x)^{-b}=(1-x)^{-(a+b)}. \]
Then we have 
\[\frac{(a)_N(b)_N}{(N!)^2}\sum_{n=0}^N\frac{(1+n)_{N-n}}{(a+n)_{N-n}}\frac{(1+N-n)_n}{(b+N-n)_n}
=\sum_{n=0}^N\frac{(a)_n}{n!}\frac{(b)_{N-n}}{(N-n)!}
=\frac{(a+b)_N}{N!}, \]
which proves the assertion. 
\end{proof}

Proposition \ref{prop:disc beta} is a finite analogue of the beta integral formula 
\[B(a,b)=\int_0^1 t^{a-1}(1-t)^{b-1}dt. \]
To justify this interpretation, we show that the function 
\[t^{a-1}(1-t)^{b-1}\]
for $0<t<1$ is approximated by the discrete function 
\[\frac{(1+n)_{N-n}}{(a+n)_{N-n}}\frac{(1+N-n)_n}{(b+N-n)_n}\]
with $t\fallingdotseq n/N$. More precisely: 

\begin{prop}
Let $t$ be a constant with $0<t<1$. 
If two integers $N$ and $n$ tend to infinity satisfying $n/N\to t$, then we have 
\begin{equation}\label{eq:dic t^a-1}
\frac{(1+n)_{N-n}}{(a+n)_{N-n}}\to t^{a-1},\quad 
\frac{(1+N-n)_{n}}{(b+N-n)_{n}}\to (1-t)^{b-1}. 
\end{equation}
\end{prop}

\begin{proof}
By symmetry, it suffices to show the first formula. 
By Stirling's formula 
\[\Gamma(x)\sim \sqrt{2\pi}x^{x-1/2}e^{-x}\quad \text{as $x\to\infty$}, \]
we have 
\begin{align}
\label{eq:Gamma quot asymp}
\frac{\Gamma(a+x)}{\Gamma(x)}
&\sim \frac{(a+x)^{a+x-1/2}e^{-(a+x)}}{x^{x-1/2}e^{-x}}\\
\notag
&=\biggl(1+\frac{a}{x}\biggr)^{a+x-1/2}\cdot e^{-a}\cdot x^{a}\sim x^{a} 
\quad \text{as $x\to\infty$}. 
\end{align}
Hence it holds that 
\[\frac{(1+n)_{N-n}}{(a+n)_{N-n}}
=\frac{\Gamma(1+N)}{\Gamma(a+N)}\frac{\Gamma(a+n)}{\Gamma(1+n)}
\sim N^{1-a}\cdot n^{a-1}\to t^{a-1}.\qedhere \]
\end{proof}

\begin{rem}
\begin{enumerate}
\item 
In \cite[Remark 2.2.3]{AAR}, it is mentioned that the equation \eqref{eq:binomial} is 
a discrete form of the beta integral formula, in the sense that one can show  
\[\frac{N!}{(a+b)_N}\sum_{n=0}^N\frac{(a)_n}{n!}\frac{(b)_{N-n}}{(N-n)!}
\to\frac{\Gamma(a+b)}{\Gamma(a)\Gamma(b)}\int_0^1 t^{a-1}(1-t)^{b-1}dt\]
as $N\to\infty$. 

\item 
The following generalization of the formula \eqref{eq:disc beta} 
is proved by the same argument: 
\begin{equation}\label{eq:multibeta}
\frac{(a_1+\cdots+a_d)_N(N!)^{d-1}}{(a_1)_N\cdots(a_d)_N}
=\sum_{\substack{n_1,\ldots,n_d\ge 0\\ n_1+\cdots+n_d=N}}
\prod_{i=1}^d\frac{(1+n_i)_{N-n_i}}{(a_i+n_i)_{N-n_i}}. 
\end{equation}
Divided by $N^{d-1}$, this gives a finite analogue of the multivariate beta integral formula: 
\[\frac{\Gamma(a_1)\cdots\Gamma(a_d)}{\Gamma(a_1+\cdots+a_d)}
=\int_{\substack{0<t_1,\ldots,t_d<1\\ t_1+\cdots+t_d=1}}
\prod_{i=1}^d t_i^{a_i-1}\ dt_1\dots dt_{d-1}\]
\end{enumerate}
\end{rem}

\section{Truncated generalized hypergeometric functions}
We define the truncated generalized hypergeometric functions by 
\[\trF{p}{q}{[N]}\biggl(\varray{a_1,\ldots,a_p\\ b_1,\ldots,b_q};z\biggr)
\coloneqq \sum_{m=0}^{N}\frac{(a_1)_m\cdots(a_p)_m}{(b_1)_m\cdots(b_q)_m m!}\frac{(N+1-m)_m}{(Nz^{-1}-m)_m}. \]
As in the case of ${}_2F_1$, this can be written in terms of the terminating hypergeometric series: 
\[\trF{p}{q}{[N]}\biggl(\varray{a_1,\ldots,a_p\\ b_1,\ldots,b_q};z\biggr)
={}_{p+1}F_{q+1}\biggl(\varray{a_1,\ldots,a_p,-N\\ b_1,\ldots,b_q,1-Nz^{-1}};1\biggr). \]

The following formula is the truncated analogue of the binomial theorem 
\[{}_1F_0\biggl(\varray{a\\ -};z\biggr)=(1-z)^{-a}. \]

\begin{prop}\label{prop:t1F0}
\begin{equation}\label{eq:t1F0}
\trF{1}{0}{[N]}\biggl(\varray{a\\ -};z\biggr)=\frac{(a+Nz^{-1}-N)_{N}}{(Nz^{-1}-N)_{N}}. 
\end{equation}
\end{prop}

\begin{proof}
By the equation \eqref{eq:binomial}, we have 
\[\frac{(Nz^{-1}-N)_N}{N!}\trF{1}{0}{[N]}\biggl(\varray{a\\ -};z\biggr)
=\sum_{m=0}^N\frac{(a)_m}{m!}\frac{(Nz^{-1}-N)_{N-m}}{(N-m)!}=\frac{(a+Nz^{-1}-N)_N}{N!}, \]
which shows the assertion. 
\end{proof}

\begin{rem}
\begin{enumerate}
\item For fixed $a\in\mathbb C$ and $0<z<1$, \eqref{eq:Gamma quot asymp} gives 
\begin{align*}
\frac{(a+Nz^{-1}-N)_{N}}{(Nz^{-1}-N)_{N}}
&=\frac{\Gamma(a+Nz^{-1})}{\Gamma(Nz^{-1})}\frac{\Gamma(Nz^{-1}-N)}{\Gamma(a+Nz^{-1}-N)}\\
&\sim (Nz^{-1})^a(Nz^{-1}-N)^{-a}=(1-z)^{-a} 
\end{align*}
as $N\to\infty$. 

\item In view of the equation 
\[\trF{1}{0}{[N]}\biggl(\varray{a\\ -};z\biggr)={}_2F_1\biggl(\varray{a,-N\\ 1-Nz^{-1}};1\biggr), \]
the formula \eqref{eq:t1F0} is equivalent to the Chu-Vandermonde theorem 
(cf.~\cite[Corollary 2.2.3]{AAR}). 
\end{enumerate}
\end{rem}

Next, we prove the finite version of the integral formula 
\begin{equation}\label{eq:gen HG int}
\begin{split}
&{}_{p+1}F_{q+1}\biggl(\varray{a,a_1,\ldots,a_p\\ b,b_1,\ldots,b_q};z\biggr)\\
&=\frac{\Gamma(b)}{\Gamma(a)\Gamma(b-a)}\int_0^1 t^{a-1}(1-t)^{b-a-1}
{}_pF_q\biggl(\varray{a_1,\ldots,a_p\\ b_1,\ldots,b_q};tz\biggr)dt. 
\end{split}
\end{equation}

\begin{thm}\label{thm:gen tHG int}
\begin{equation}\label{eq:gen tHG int}
\begin{split}
&\trF{p+1}{q+1}{[N]}\biggl(\varray{a,a_1,\ldots,a_p\\ b,b_1,\ldots,b_q};z\biggr)\\
&=\frac{(a)_{N} (b-a)_{N}}{(b)_{N} N!}
\sum_{n=0}^{N}\frac{(1+n)_{N-n}}{(a+n)_{N-n}}\frac{(1+N-n)_{n}}{(b-a+N-n)_{n}}
\trF{p}{q}{[n]}\biggl(\varray{a_1,\ldots,a_p\\ b_1,\ldots,b_q};\frac{n}{N}z\biggr). 
\end{split}
\end{equation}
When $N=0$, we regard $\trF{p}{q}{[n]}(\ldots;\tfrac{n}{N}z)$ as $1$. 
\end{thm}

\begin{proof}
We start from 
\begin{align*}
&\frac{(b)_{N} N!}{(a)_{N}(b-a)_{N}}
\trF{p+1}{q+1}{[N]}\biggl(\varray{a,a_1,\ldots,a_p\\ b,b_1,\ldots,b_q};z\biggr)\\
&=\sum_{m=0}^{N}\frac{(b+m)_{N-m} N!}{(a+m)_{N-m}(b-a)_{N}}
\frac{(a_1)_m\cdots(a_p)_m}{(b_1)_m\cdots(b_q)_m m!}\frac{(N+1-m)_m}{(Nz^{-1}-m)_m}\\
&=\sum_{m=0}^{N}\frac{(b+m)_{N-m}(N-m)!}{(a+m)_{N-m}(b-a)_{N-m}}\\
&\mspace{60mu}\cdot\frac{(1+N-m)_m}{(b-a+N-m)_m}
\frac{(a_1)_m\cdots(a_p)_m}{(b_1)_m\cdots(b_q)_m m!}\frac{(N+1-m)_m}{(Nz^{-1}-m)_m}. 
\end{align*}
By Proposition \ref{prop:disc beta}, we have 
\[\frac{(b+m)_{N-m}(N-m)!}{(a+m)_{N-m}(b-a)_{N-m}}
=\sum_{l=0}^{N-m}\frac{(1+l)_{N-m-l}}{(a+m+l)_{N-m-l}}\frac{(1+N-m-l)_l}{(b-a+N-m-l)_l}, \]
which implies 
\begin{align*}
&\frac{(b)_{N} N!}{(a)_{N}(b-a)_{N}}
\trF{p+1}{q+1}{[N]}\biggl(\varray{a,a_1,\ldots,a_p\\ b,b_1,\ldots,b_q};z\biggr)\\
&=\sum_{m=0}^{N}\Biggl(\sum_{l=0}^{N-m}\frac{(1+l)_{N-m-l}}{(a+m+l)_{N-m-l}}
\frac{(1+N-m-l)_l}{(b-a+N-m-l)_l}\Biggr)\\
&\mspace{60mu}\cdot\frac{(1+N-m)_m}{(b-a+N-m)_m}
\frac{(a_1)_m\cdots(a_p)_m}{(b_1)_m\cdots(b_q)_m m!}\frac{(N+1-m)_m}{(Nz^{-1}-m)_m}. 
\end{align*}
By setting $n=l+m$ and changing the order of summation, we see that this equals 
\begin{align*}
&\sum_{n=0}^{N}\frac{(1+N-n)_{n}}{(a+n)_{N-n}(b-a+N-n)_{n}}
\sum_{m=0}^{n}\frac{(a_1)_m\cdots(a_p)_m}{(b_1)_m\cdots(b_q)_m m!}
\frac{(1+n-m)_{N-n+m}}{(Nz^{-1}-m)_m}\\
&=\sum_{n=0}^{N}\frac{(1+n)_{N-n}}{(a+n)_{N-n}}\frac{(1+N-n)_{n}}{(b-a+N-n)_{n}}
\sum_{m=0}^{n}\frac{(a_1)_m\cdots(a_p)_m}{(b_1)_m\cdots(b_q)_m m!}
\frac{(n+1-m)_{m}}{(n(\frac{n}{N}z)^{-1}-m)_m}\\
&=\sum_{n=0}^{N}\frac{(1+n)_{N-n}}{(a+n)_{N-n}}\frac{(1+N-n)_{n}}{(b-a+N-n)_{n}}
\trF{p}{q}{[n]}\biggl(\varray{a_1,\ldots,a_p\\ b_1,\ldots,b_q};\frac{n}{N}z\biggr). 
\end{align*}
Thus the assertion is proved. 
\end{proof}

In particular, we obtain the following: 

\begin{thm}\label{thm:tHG int}
\begin{equation}\label{eq:tHG int}
\begin{split}
\trF{2}{1}{[N]}&\biggl(\varray{a,b\\ c};z\biggr)\\
&=\frac{(a)_{N} (c-a)_{N}}{(c)_{N} N!}
\sum_{n=0}^{N}\frac{(1+n)_{N-n}}{(a+n)_{N-n}}\frac{(1+N-n)_{n}}{(c-a+N-n)_{n}}
\frac{(b+Nz^{-1}-n)_{n}}{(Nz^{-1}-n)_{n}}. 
\end{split}
\end{equation}
\end{thm}

\begin{proof}
This follows from Theorem \ref{thm:gen tHG int} and Proposition \ref{prop:t1F0}. 
\end{proof}

\begin{rem}
Theorem \ref{thm:tHG int} is written as 
\begin{align*}
{}_3F_2&\biggl(\varray{a,b,-N\\ c,1-Nz^{-1}};1\biggr)\\
&=\frac{(a)_{N} (c-a)_{N}}{(c)_{N} N!}
\sum_{n=0}^{N}\frac{(1+n)_{N-n}}{(a+n)_{N-n}}\frac{(1+N-n)_{n}}{(c-a+N-n)_{n}}
\frac{(b+Nz^{-1}-n)_{n}}{(Nz^{-1}-n)_{n}}. 
\intertext{The right-hand side can be further rewritten as follows: }
&=\frac{(c-a)_{N}}{(c)_{N}}
\sum_{n=0}^{N}\frac{(a)_{n}}{n!}\frac{(-N)_{n}}{(1+a-c-N)_{n}}\frac{(1-Nz^{-1}-b)_{n}}{(1-Nz^{-1})_{n}}\\
&=\frac{(c-a)_{N}}{(c)_{N}}{}_3F_2\biggl(\varray{a,1-Nz^{-1}-b,-N\\ 1+a-c-N,1-Nz^{-1}};1\biggr). 
\end{align*}
This coincides with the following equation (up to substitution of letters)
\begin{equation}\label{eq:3F2 transformation}
{}_3F_2\biggl(\varray{a,b,-N\\ d,e};1\biggr)
=\frac{(e-a)_N}{(e)_N}{}_3F_2\biggl(\varray{a,d-b,-N\\ d,a+1-N-e};1\biggr). 
\end{equation}
This transformation formula \eqref{eq:3F2 transformation} appears 
in the proof of \cite[Corollary 3.3.4]{AAR}, 
and also follows from \cite[Theorem 2.4.4]{AAR} by letting $a\to -N$.  
\end{rem}

Considering the vast amount of research on hypergeometric functions, 
it would not be surprising that if Theorem \ref{thm:gen tHG int} had already appeared 
in the literature, although the author has been unable to locate a reference. 
Here we write it explicitly in terms of the terminating hypergeometric series 
(with certain relabeling of letters): 
\begin{multline*}
{}_{p+1}F_{q}\biggl(\varray{a_1,\ldots,a_p,-N\\ b_1,\ldots,b_q};1\biggr)\\
=\frac{(a_p)_N(b_q-a_p)_N}{(b_q)_N N!}
\sum_{n=0}^N\frac{(1+n)_{N-n}}{(a_p+n)_{N-n}}\frac{(1+N-n)_n}{(b_q-a_p+N-n)_n}\\
\cdot{}_{p}F_{q-1}\biggl(\varray{a_1,\ldots,a_{p-1},-n\\ b_1,\ldots,b_{q-1}};1\biggr). 
\end{multline*}
After some simple manipulation, it is also written as 
\begin{equation}\label{eq:gen tHG int term}
\begin{split}
{}_{p+1}F_{q}&\biggl(\varray{a_1,\ldots,a_p,-N\\ b_1,\ldots,b_q};1\biggr)\\
&=\sum_{n=0}^N\binom{N}{n}\frac{(a_p)_n(b_q-a_p)_{N-n}}{(b_q)_N}
{}_{p}F_{q-1}\biggl(\varray{a_1,\ldots,a_{p-1},-n\\ b_1,\ldots,b_{q-1}};1\biggr). 
\end{split}
\end{equation}
We also record the following formula, which may be interesting in its own right: 

\begin{cor}
\begin{align*}
{}_{p+1}F_p&\biggl(\varray{a_1,\ldots,a_p,-N\\ b_1,\ldots,b_p};1\biggr)\\
&=\sum_{0=n_0\le n_1\le\cdots\le n_{p-1}\le n_p=N}
\prod_{j=1}^p\binom{n_j}{n_{j-1}}\frac{(a_j)_{n_{j-1}}(b_j-a_j)_{n_j-n_{j-1}}}{(b_j)_{n_j}}. 
\end{align*}
\end{cor}

\begin{proof}
It follows by using the equation \eqref{eq:gen tHG int term} inductively. 
\end{proof}

\section{Applications of Theorem \ref{thm:tHG int}}
By using the integral representation \eqref{eq:HG int} of the hypergeometric function, 
one can deduce the following results: 
\begin{itemize}
\item The hypergeometric theorem of Gauss: 
\[{}_2F_1\biggl(\varray{a,b\\ c};1\biggr)=\frac{\Gamma(c)\Gamma(c-a-b)}{\Gamma(c-a)\Gamma(c-b)}. \]
\item The transformation formulas of Pfaff and Euler: 
\begin{align*}
{}_2F_1\biggl(\varray{a,b\\ c};z\biggr)
&=(1-z)^{-b}{}_2F_1\biggl(\varray{c-a,b\\ c};\frac{z}{z-1}\biggr)\\
&=(1-z)^{c-a-b}{}_2F_1\biggl(\varray{c-a,c-b\\ c};z\biggr). 
\end{align*}
\end{itemize}
Now we shall give the finite analogues of them. 

\begin{cor}\label{cor:tHG theorem}
\[\trF{2}{1}{[N]}\biggl(\varray{a,b\\ c};\frac{N}{N+c-a-b}\bigg)
=\frac{(c-a)_{N}(c-b)_{N}}{(c)_{N}(c-a-b)_{N}}. \]
\end{cor}

\begin{proof}
By Theorem \ref{thm:tHG int} and Proposition \ref{prop:disc beta}, 
one sees that the left-hand side is equal to 
\begin{align*}
&\frac{(a)_{N}(c-a)_{N}}{(c)_{N} N!}
\sum_{n=0}^{N}\frac{(1+n)_{N-n}}{(a+n)_{N-n}}
\frac{(1+N-n)_{n}}{(c-a+N-n)_{n}}\frac{(c-a+N-n)_n}{(c-a-b+N-n)_n}\\
&=\frac{(a)_{N}(c-a)_{N}}{(c)_{N} N!}
\sum_{n=0}^{N}\frac{(1+n)_{N-n}}{(a+n)_{N-n}}\frac{(1+N-n)_{n}}{(c-a-b+N-n)_n}\\
&=\frac{(a)_{N}(c-a)_{N}}{(c)_{N} N!}
\frac{(c-b)_{N} N!}{(a)_{N}(c-a-b)_{N}}.  
\end{align*}
This is equal to the right-hand side. 
\end{proof}

\begin{rem}
Corollary \ref{cor:tHG theorem}, written as 
\[{}_3F_2\biggl(\varray{a,b,-N\\ c,1-N+a+b-c};1\biggr)
=\frac{(c-a)_N(c-b)_N}{(c)_N(c-a-b)_N}, \]
is exactly the Pfaff-Saalsch\"{u}tz theorem (cf.~\cite[Theorem 2.2.6]{AAR}). 
\end{rem}

\begin{cor}\label{cor:tHG Pfaff Euler}
\begin{align*}
\trF{2}{1}{[N]}\biggl(\varray{a,b\\ c};z\biggr)
&=\frac{(b+Nz^{-1}-N)_{N}}{(Nz^{-1}-N)_{N}}
\trF{2}{1}{[N]}\biggl(\varray{c-a,b\\ c};\frac{N}{N-Nz^{-1}+1-b}\biggr) \\
&=\frac{(a+b-c+Nz^{-1}-N)_{N}}{(Nz^{-1}-N)_{N}}
\trF{2}{1}{[N]}\biggl(\varray{c-a,c-b\\ c};\frac{N}{Nz^{-1}+a+b-c}\biggr). 
\end{align*}
\end{cor}

\begin{proof}
Put $w=Nz^{-1}$. By setting $k=N-n$ in the expression \eqref{eq:tHG int}, we have 
\begin{align*}
\trF{2}{1}{[N]}\biggl(\varray{a,b\\ c};z\biggr)
=\frac{(a)_{N} (c-a)_{N}}{(c)_{N} N!}
\sum_{k=0}^{N}\frac{(1+N-k)_{k}}{(a+N-k)_{k}}&\frac{(1+k)_{N-k}}{(c-a+k)_{N-k}}\\
&\cdot\frac{(b+w-N+k)_{N-k}}{(w-N+k)_{N-k}}. 
\end{align*}
Moreover, since 
\begin{align*}
\frac{(b+w-N+k)_{N-k}}{(w-N+k)_{N-k}}
&=\frac{(b+w-N)_{N}}{(w-N)_{N}}\frac{(w-N)_{k}}{(b+w-N)_{k}}\\
&=\frac{(b+w-N)_{N}}{(w-N)_{N}}\frac{(N-w-k+1)_{k}}{(N-w-b-k+1)_{k}}\\
&=\frac{(b+w-N)_{N}}{(w-N)_{N}}\frac{(b+(N-w+1-b)-k)_{k}}{((N-w+1-b)-k)_{k}}, 
\end{align*}
we obtain 
\begin{align*}
\trF{2}{1}{[N]}&\biggl(\varray{a,b\\ c};z\biggr)\\
&=\frac{(b+w-N)_{N}}{(w-N)_{N}}\frac{(a)_{N} (c-a)_{N}}{(c)_{N} N!}\\
&\quad\cdot\sum_{k=0}^{N}\frac{(1+N-k)_{k}}{(a+N-k)_{k}}\frac{(1+k)_{N-k}}{(c-a+k)_{N-k}}
\frac{(b+(N-w+1-b)-k)_{k}}{((N-w+1-b)-k)_{k}}\\
&=\frac{(b+w-N)_{N}}{(w-N)_{N}}
\trF{2}{1}{[N]}\biggl(\varray{c-a,b\\ c};\frac{N}{N-w+1-b}\biggr). 
\end{align*}
Here we have used \eqref{eq:tHG int} again. Thus the first formula is proved. 

The first formula implies that 
\begin{align*}
\trF{2}{1}{[N]}&\biggl(\varray{c-a,b\\ c};\frac{N}{N-w+1-b}\biggr)\\
&=\frac{(c-a+(N-w+1-b)-N)_{N}}{((N-w+1-b)-N)_{N}}\\
&\quad\cdot\trF{2}{1}{[N]}\biggl(\varray{c-a,c-b\\ c};
\frac{N}{N-(N-w+1-b)+1-(c-a)}\biggr)\\
&=\frac{(a+b-c+w-N)_{N}}{(b+w-N)_{N}}
\trF{2}{1}{[N]}\biggl(\varray{c-a,c-b\\ c};\frac{N}{w+a+b-c}\biggr). 
\end{align*}
By substituting this into the first formula, we prove the second one. 
\end{proof}

\begin{rem}
The first formula of Corollary \ref{cor:tHG Pfaff Euler} can be written as 
\[{}_3F_2\biggl(\varray{a,b,-N\\ c,1-Nz^{-1}};1\biggr)
=\frac{(1-Nz^{-1}-b)_N}{(1-Nz^{-1})_N}{}_3F_2\biggl(\varray{c-a,b,-N\\ c,b+Nz^{-1}-N};1\biggr). \]
This is also equivalent to \eqref{eq:3F2 transformation}. 
It is interesting that the finite analogues of 
the integral representation of ${}_2F_1$ and Pfaff's transformation are equivalent to 
the same equation \eqref{eq:3F2 transformation}. 
\end{rem}

\section{Ohno-Zagier type formula}
Let us recall the formula obtained by Ohno and Zagier \cite{OZ}. 
For integers $k,r,h>0$, let $I_0(k,r,h)$ denote the set of 
admissible indices of weight $k$, depth $r$ and height $h$, that is, 
\[I_0(k,r,h)\coloneqq \biggl\{(k_1,\ldots,k_r)\in\Z_{>0}^r\biggm| 
\begin{array}{l} 
k_r>1, k_1+\cdots+k_r=k, \\
\#\{i\mid k_i>1\}=h
\end{array}\biggr\}. \]
Note that $I_0(k,r,h)$ is empty unless $k\ge r+h$ and $r\ge h$. 
For any $\bk=(k_1,\ldots,k_r)\in \Z_{>0}^r$, we define the multiple polylogarithm by 
\[\Li_\bk(z)\coloneqq\sum_{0<m_1<\cdots<m_r}\frac{z^{m_r}}{m_1^{k_1}\cdots m_r^{k_r}}. \]
Then we consider the sum of the multiple polylogarithms of given weight, depth and height 
\[G_0(k,r,h;z)\coloneqq \sum_{\bk\in I_0(k,r,h)}\Li_\bk(z), \]
and their generating function 
\[\Phi_0(X,Y,Z;z)\coloneqq \sum_{k,r,h>0}G_0(k,r,h;z)X^{k-r-h}Y^{r-h}Z^{h-1}. \]

\begin{thm}[{\cite[p.~485]{OZ}}]
We have 
\begin{equation}\label{eq:OZ}
\Phi_0(X,Y,Z;z)=\frac{1}{Z-XY}\biggl\{{}_2F_1\biggl(\varray{\alpha-X,\beta-X\\ 1-X};z\biggr)-1\biggr\}, 
\end{equation}
where $\alpha$ and $\beta$ satisfy the relations 
\[\alpha+\beta=X+Y,\quad \alpha\beta=Z. \]
\end{thm}

Now we consider the truncation. For a positive integer $N$, define  
\begin{equation}\label{eq:Phi_0^(N) def}
\begin{split}
\Li_\bk^{(N)}(z)&\coloneqq\sum_{0<m_1<\cdots<m_r<N}\frac{1}{m_1^{k_1}\cdots m_r^{k_r}}
\frac{(N-m_r)_{m_r}}{(Nz^{-1}-m_r)_{m_r}}, \\
G_0^{(N)}(k,r,h;z)&\coloneqq \sum_{\bk\in I_0(k,r,h)}\Li^{(N)}_\bk(z), \\
\Phi_0^{(N)}(X,Y,Z;z)&\coloneqq \sum_{k,r,h>0}G^{(N)}_0(k,r,h;z)X^{k-r-h}Y^{r-h}Z^{h-1}. 
\end{split}
\end{equation}
Note that $\Li_\bk^{(N)}(z)$ is obtained by setting $x_1=\cdots=x_r=z^{-1}$ in \eqref{eq:HMS}. 
As mentioned in Section 1, we also consider the truncated hypergeometric function 
\[\trF{2}{1}{(N)}\biggl(\varray{a,b\\ c};z\biggr)
\coloneqq \sum_{m=0}^{N-1}\frac{(a)_m(b)_m}{(c)_m m!}\frac{(N-m)_m}{(Nz^{-1}-m)_m}
=\trF{2}{1}{[N-1]}\biggl(\varray{a,b\\ c};\frac{N-1}{N}z\biggr). \]

\begin{thm}\label{thm:tOZ}
We have 
\begin{equation}\label{eq:tOZ}
\Phi_0^{(N)}(X,Y,Z;z)=\frac{1}{Z-XY}
\biggl\{\trF{2}{1}{(N)}\biggl(\varray{\alpha-X,\beta-X\\ 1-X};z\biggr)-1\biggr\}, 
\end{equation}
where $\alpha$ and $\beta$ satisfy the relations 
\[\alpha+\beta=X+Y,\quad \alpha\beta=Z. \]
\end{thm}

This finite analogue of the Ohno-Zagier formula is indeed a \emph{consequence} 
of the original formula \eqref{eq:OZ}; 
it is obtained by applying the linear map on power series 
\[\sum_{m=0}^\infty c_mz^m\longmapsto \sum_{m=0}^{N-1}c_m\frac{(N-m)_m}{(Nz^{-1}-m)_m} \]
to both sides of \eqref{eq:OZ}.

In the following, we give a direct proof of \eqref{eq:tOZ}. 
The same argument also provides another proof of the original \eqref{eq:OZ}, 
which does not use the differential equation. 

\begin{proof}
First we claim that 
\begin{align}
\label{eq:Phi_0^(N)}
\Phi_0^{(N)}(X,Y,Z;z)=\sum_{0<n<N}
&\prod_{0<m<n}\biggl(1+\frac{Y}{m}+\frac{Z}{m^2}+\frac{XZ}{m^3}+\frac{X^2Z}{m^4}+\cdots\biggr)\\
\notag 
&\times\biggl(\frac{1}{n^2}+\frac{X}{n^3}+\frac{X^2}{n^4}+\cdots\biggr)
\frac{(N-n)_n}{(Nz^{-1}-n)_n}. 
\end{align}
Indeed, expanding the product in the right-hand side, we obtain the sum of the terms 
\[\frac{X^{k-r-h}Y^{r-h}Z^{h-1}}{m_1^{k_1}\cdots m_r^{k_r}}\frac{(N-m_r)_{m_r}}{(Nz^{-1}-m_r)_{m_r}},\]
where $m_1,\ldots,m_r$ run through integers with $0<m_1<\cdots<m_r<N$, 
$k_1,\ldots,k_r$ run through integers with $k_1,\ldots,k_{r-1}\ge 1$, $k_r\ge 2$, 
and $k$ and $h$ denote the numbers $k_1+\cdots+k_r$ and $\#\{i\mid k_i>1\}$ respectively. 
The whole sum equals $\Phi_0^{(N)}(X,Y,Z;z)$ by definition. 

Since 
\begin{align*}
1+\frac{Y}{m}+\frac{Z}{m^2}+\frac{XZ}{m^3}+\frac{X^2Z}{m^4}+\cdots
&=1+\frac{Y}{m}+\frac{Z}{m(m-X)}\\
&=\frac{m^2+(Y-X)m+Z-XY}{m(m-X)}\\
&=\frac{(m+\alpha-X)(m+\beta-X)}{m(m-X)}
\end{align*}
and 
\[\frac{1}{n^2}+\frac{X}{n^3}+\frac{X^2}{n^4}+\cdots=\frac{1}{n(n-X)}, \]
we obtain from \eqref{eq:Phi_0^(N)} that 
\begin{align*}
\Phi_0^{(N)}(X,Y,Z;z)
&=\sum_{0<n<N}\prod_{0<m<n}\frac{(m+\alpha-X)(m+\beta-X)}{m(m-X)}\cdot
\frac{1}{n(n-X)}\frac{(N-n)_n}{(Nz^{-1}-n)_n}\\
&=\sum_{0<n<N}\frac{(1+\alpha-X)_{n-1} (1+\beta-X)_{n-1}}{(1-X)_n n!}
\frac{(N-n)_n}{(Nz^{-1}-n)_n}\\
&=\frac{1}{(\alpha-X)(\beta-X)}\sum_{0<n<N}
\frac{(\alpha-X)_n (\beta-X)_n}{(1-X)_n n!}\frac{(N-n)_n}{(Nz^{-1}-n)_n}\\
&=\frac{1}{Z-XY}\biggl\{\trF{2}{1}{(N)}\biggl(\varray{\alpha-X,\beta-X\\ 1-X};z\biggr)-1\biggr\}. \qedhere
\end{align*}
\end{proof}

By combining Theorem \ref{thm:tOZ} with Corollary \ref{cor:tHG theorem}, we obtain: 

\begin{cor}\label{cor:tOZ sp}
\begin{equation}\label{eq:tOZ sp}
\begin{split}
\Phi_0^{(N)}&\biggl(X,Y,Z;\frac{N}{N-Y}\biggr)\\
&=\frac{1}{Z-XY}\biggl\{\frac{(1-\alpha)_{N-1}(1-\beta)_{N-1}}{(1-X)_{N-1}(1-Y)_{N-1}}-1\biggr\}\\
&=\frac{1}{Z-XY}\Biggl\{\exp\Biggl(
\sum_{k=2}^\infty\frac{\zeta^{(N)}(k)}{k}(X^k+Y^k-\alpha^k-\beta^k)\Biggr)-1\Biggr\}. 
\end{split}
\end{equation}
\end{cor}

Here $\zeta^{(N)}(k)\coloneqq \sum_{n=1}^{N-1}\frac{1}{n^k}$ denotes the truncated zeta value, 
and the second equation is seen as follows: 
\begin{align*}
\log\frac{(1-\alpha)_{N-1}(1-\beta)_{N-1}}{(1-X)_{N-1}(1-Y)_{N-1}}
&=\sum_{n=1}^{N-1}\log\frac{(1-\frac{\alpha}{n})(1-\frac{\beta}{n})}{(1-\frac{X}{n})(1-\frac{Y}{n})}\\
&=\sum_{k=1}^\infty\sum_{n=1}^{N-1}\frac{1}{k}
\biggl\{\biggl(\frac{X}{n}\biggr)^k+\biggl(\frac{Y}{n}\biggr)^k
-\biggl(\frac{\alpha}{n}\biggr)^k-\biggl(\frac{\beta}{n}\biggr)^k\biggr\}\\
&=\sum_{k=2}^\infty\frac{\zeta^{(N)}(k)}{k}(X^k+Y^k-\alpha^k-\beta^k). 
\end{align*}

The equation \eqref{eq:tOZ sp} is a finite analogue of \cite[Theorem 1]{OZ}:  
\[\begin{split}
\Phi_0(X,Y,Z;1)
&=\frac{1}{Z-XY}\biggl\{\frac{\Gamma(1-X)\Gamma(1-Y)}{\Gamma(1-\alpha)\Gamma(1-\beta)}-1\biggr\}\\
&=\frac{1}{Z-XY}\Biggl\{\exp\Biggl(
\sum_{k=2}^\infty\frac{\zeta(k)}{k}(X^k+Y^k-\alpha^k-\beta^k)\Biggr)-1\Biggr\}. 
\end{split}\]
From this equation, many relations among multiple zeta values are deduced. 
For instance, the symmetry in $X$ and $Y$ implies the relation 
\[\sum_{\bk\in I_0(k,r,h)}\zeta(\bk)=\sum_{\bk\in I_0(k,k-r,h)}\zeta(\bk). \]
In fact, this relation is also a consequence of the duality relation, 
which shows that there is a one-to-one correspondence between the sets $I_0(k,r,h)$ and $I_0(k,k-r,h)$ 
such that the corresponding multiple zeta values are equal. 

To examine a similar interpretation for the equation \eqref{eq:tOZ sp}, let us expand the left-hand side 
as a power series in $X$, $Y$ and $Z$. 
By definition, we have 
\[\Phi_0^{(N)}\biggl(X,Y,Z;\frac{N}{N-Y}\biggr)
=\sum_{k,r,h>0}X^{k-r-h}Y^{r-h}Z^{h-1}\sum_{\bk\in I_0(k,r,h)}\Li_\bk^{(N)}\biggl(\frac{N}{N-Y}\biggr) \]
and 
\begin{align*}
\Li_\bk^{(N)}\biggl(\frac{N}{N-Y}\biggr)
&=\sum_{0<m_1<\cdots<m_r<N}\frac{1}{m_1^{k_1}\cdots m_r^{k_r}}\frac{(N-m_r)_{m_r}}{(N-Y-m_r)_{m_r}}\\
&=\sum_{0<m_1<\cdots<m_r<N}\frac{1}{m_1^{k_1}\cdots m_r^{k_r}}
\prod_{n=1}^{m_r}\biggl(1-\frac{Y}{N-n}\biggr)^{-1}\\
&=\sum_{0<m_1<\cdots<m_r<N}\frac{1}{m_1^{k_1}\cdots m_r^{k_r}}
\sum_{\substack{l\ge 0\\ 0<n_1\le\cdots\le n_l\le m_r}}
\frac{Y^l}{(N-n_1)\cdots(N-n_l)} 
\end{align*}
for $\bk=(k_1,\ldots,k_r)$. 
If we put 
\[\tilde{\zeta}^{(N)}(k_1,\ldots,k_r;l)\coloneqq 
\sum_{0<m_1<\cdots<m_r<N}\frac{1}{m_1^{k_1}\cdots m_r^{k_r}}
\sum_{0<n_1\le\cdots\le n_l\le m_r}\frac{1}{(N-n_1)\cdots(N-n_l)} \]
and 
\[\tilde{I}_0(k,q,h)\coloneqq 
\bigcup_{r\ge 1}
\biggl\{(k_1,\ldots,k_r;l)\in\Z_{>0}^r\times\Z_{\ge 0}\biggm| \begin{array}{l}
k_r>1,\ k_1+\cdots+k_r+l=k,\\
r+l=q,\ \#\{i\mid k_i>1\}=h\end{array}\biggr\}, \]
then we have: 

\begin{prop}\label{prop:Phi_0^(N)}
\[\Phi_0^{(N)}\biggl(X,Y,Z;\frac{N}{N-Y}\biggr)
=\sum_{k,q,h>0}X^{k-q-h}Y^{q-h}Z^{h-1}\sum_{(\bk;l)\in \tilde{I}_0(k,q,h)}\tilde{\zeta}^{(N)}(\bk;l). \]
\end{prop}

Corollary \ref{cor:tOZ sp} and Proposition \ref{prop:Phi_0^(N)} immediately imply the following: 

\begin{cor}\label{cor:symmetry}
\begin{enumerate}
\item 
For each $k,q,h>0$, there exists a polynomial $P_{k,q,h}(Z_2,Z_3,\ldots)$, 
independent of $N$, such that the identity 
\[\sum_{(\bk;l)\in \tilde{I}_0(k,q,h)}\tilde{\zeta}^{(N)}(\bk;l)
=P_{k,q,h}(\zeta^{(N)}(2),\zeta^{(N)}(3),\ldots)\]
holds. If we assign degree $j$ to the letter $Z_j$, 
then the polynomial $P_{k,q,h}(Z_2,Z_3,\ldots)$ is homogeneous of degree $k$. 
 
\item We have the symmetry 
\begin{equation}\label{eq:symmetry}
\sum_{(\bk;l)\in \tilde{I}_0(k,q,h)}\tilde{\zeta}^{(N)}(\bk;l)
=\sum_{(\bk;l)\in \tilde{I}_0(k,k-q,h)}\tilde{\zeta}^{(N)}(\bk;l). 
\end{equation}
\end{enumerate}
\end{cor}

\begin{rem}\label{rem:zeta tilde}
\begin{enumerate}
\item Unlike the case of the multiple zeta values (in which the duality relation holds), 
there is no one-to-one correspondence of the values $\tilde{\zeta}^{(N)}(\bk;l)$ 
which refines the relation \eqref{eq:symmetry} in general. 
Indeed, the sets $\tilde{I}_0(k,q,h)$ and $\tilde{I}_0(k,k-q,h)$ 
do not necessarily have the same cardinality, e.g., 
\[\tilde{I}_0(3,1,1)=\{(3;0)\},\quad 
\tilde{I}_0(3,2,1)=\{(1,2;0),(2;1)\}. \]
\item After the announcement on arXiv of the first version of this paper, 
Shin-ichiro Seki informed the author that Takumi Maesaka had interpreted 
Proposition \ref{prop:Phi_0^(N)} in terms of the values $\zeta^{\diamond}_N(\bk)$ 
introduced by Hirose-Maesaka-Seki-Watanabe \cite{HMSW}. 
In fact, one can show the identity 
\[\sum_{(\bk;l)\in \tilde{I}_0(k,q,h)}\tilde{\zeta}^{(N)}(\bk;l)
=\sum_{\bk\in I_0(k,q,h)}\zeta^\diamond_N(\bk) \]
by an elementary combinatorial argument. 
Then the duality relation for $\zeta^\diamond_N(\bk)$ 
(see the sentence after Corollary 3.10 of \cite{HMSW}) refines the symmetry \eqref{eq:symmetry}. 
\end{enumerate}
\end{rem}

We point out that the definition of $\tilde{\zeta}^{(N)}(\bk;l)$ is 
somewhat similar to the series expression of 
the Arakawa-Kaneko multiple zeta value 
\begin{align*}
\xi(\bk;l+1)
&\coloneqq\frac{1}{\Gamma(l+1)}\int_0^\infty
\frac{\Li_{\bk}(1-e^{-x})}{e^x-1}x^{l}dx\\
&=\sum_{\substack{0<m_1<\cdots<m_r\\ 0<n_1\le\cdots\le n_{l}\le m_r}}
\frac{1}{m_1^{k_1}\cdots m_r^{k_r+1}}\frac{1}{n_1\cdots n_{l}} 
\end{align*}
(the latter expression follows from the integral-series identity \cite{KY}). 
Based on this expression, we may define the truncated 
Arakawa-Kaneko multiple zeta value as 
\[\xi^{(N)}(\bk;l+1)
\coloneqq\sum_{\substack{0<m_1<\cdots<m_r<N\\ 0<n_1\le\cdots\le n_{l}\le m_r}}
\frac{1}{m_1^{k_1}\cdots m_r^{k_r+1}}\frac{1}{n_1\cdots n_{l}}. \]
The similarity becomes clearer if we assume $N=p$ is a prime 
and consider the value modulo $p$, 
as in the setting of the finite multiple zeta values 
(cf.\ \cite[\S7]{K}): 
\[\xi^{(p)}(\bk;l+1)\equiv
(-1)^{l}\,\tilde{\zeta}^{(p)}(k_1,\ldots,k_{r-1},k_r+1;l)\pmod{p}. \]

\section{Two-variable case}
Since there are hypergeometric functions of multiple variables 
which have both series and integral representations, 
one may expect their finite analogues. 
In this section, we present some results in this direction 
in the case of two variables. 

First recall the definitions of Appell's hypergeometric series: 
\begin{align*}
F_1\biggl(\varray{a,b,b'\\ c};z,w\biggr)
&=\sum_{m,n=0}^\infty
\frac{(a)_{m+n}(b)_m(b')_n}{(c)_{m+n}}\frac{z^m w^n}{m!n!},\\
F_2\biggl(\varray{a,b,b'\\ c,c'};z,w\biggr)
&=\sum_{m,n=0}^\infty
\frac{(a)_{m+n}(b)_m(b')_n}{(c)_m(c')_n}\frac{z^mw^n}{m!n!},\\
F_3\biggl(\varray{a,a',b,b'\\ c};z,w\biggr)
&=\sum_{m,n=0}^\infty
\frac{(a)_m(a')_n(b)_m(b')_n}{(c)_{m+n}}\frac{z^mw^n}{m!n!},\\
F_4\biggl(\varray{a,b\\ c,c'};z,w\biggr)
&=\sum_{m,n=0}^\infty
\frac{(a)_{m+n}(b)_{m+n}}{(c)_m(c')_n}\frac{z^mw^n}{m!n!}. 
\end{align*}
As a generalization of these, 
we consider the following series: 
\begin{equation}\label{eq:HGF 2var}
\trF{p}{q}{}\biggl(\varray{\ba_1,\ldots,\ba_p\\ \bb_1,\ldots,\bb_q};
z,w\biggr)
\coloneqq \sum_{m,n=0}^{\infty}
\frac{(\ba_1)_{m,n}\cdots(\ba_p)_{m,n}}
{(\bb_1)_{m,n}\cdots (\bb_q)_{m,n}}
\frac{z^mw^n}{m!n!}. 
\end{equation}
Here each $\ba_i$ (resp.\ $\bb_j$) is either 
a single number $a_i$ (resp.\ $b_j$) or 
a pair of numbers $(a_i,a_i')$ (resp.\ $(b_j,b_j')$), 
and we set 
\[(\ba)_{m,n}\coloneqq\begin{cases}
(a)_{m+n}& \text{if $\ba=a$ is a single number}, \\
(a)_m(a')_n& \text{if $\ba=(a,a')$ is a pair of numbers}.
\end{cases}\]
Moreover, we introduce two types of finite truncation of 
the series \eqref{eq:HGF 2var}: 
\begin{align*}
&\trF{p}{q}{[N]}
\biggl(\varray{\ba_1,\ldots,\ba_p\\ \bb_1,\ldots,\bb_q};z,w\biggr)\\
&\coloneqq \sum_{\substack{m,n\ge 0\\ m+n\le N}}
\frac{(\ba_1)_{m,n}\cdots(\ba_p)_{m,n}}
{(\bb_1)_{m,n}\cdots (\bb_q)_{m,n}m!n!}
\frac{(N+1-m-n)_{m+n}}{(Nz^{-1}-m)_m(Nw^{-1}-n)_n},\\
&\trF{p}{q}{[M,N]}
\biggl(\varray{\ba_1,\ldots,\ba_p\\ \bb_1,\ldots,\bb_q};z,w\biggr)\\
&\coloneqq \sum_{\substack{0\le m\le M\\ 0\le n\le N}}
\frac{(\ba_1)_{m,n}\cdots(\ba_p)_{m,n}}
{(\bb_1)_{m,n}\cdots (\bb_q)_{m,n}m!n!}
\frac{(M+1-m)_m}{(Mz^{-1}-m)_m}\frac{(N+1-n)_n}{(Nw^{-1}-n)_n}. 
\end{align*}
As the two-variable version of Theorem \ref{thm:gen tHG int}, 
we have the following: 

\begin{thm}\label{thm:two var}
\begin{enumerate}
\item For single numbers $a$ and $b$, 
\begin{equation}\label{eq:2var 1}
\begin{split}
&\trF{p+1}{q+1}{[N]}\biggl(\varray{a,\ba_1,\ldots,\ba_p\\ 
b,\bb_1,\ldots,\bb_q};z,w\biggr)\\
&=\frac{(a)_{N} (b-a)_{N}}{(b)_{N} N!}\\
&\hspace{20pt}\times
\sum_{n=0}^{N}\frac{(1+n)_{N-n}}{(a+n)_{N-n}}
\frac{(1+N-n)_{n}}{(b-a+N-n)_{n}}
\trF{p}{q}{[n]}\biggl(\varray{\ba_1,\ldots,\ba_p\\ 
\bb_1,\ldots,\bb_q};\frac{n}{N}z,\frac{n}{N}w\biggr). 
\end{split}
\end{equation}
\item For a pair of numbers $(a,a')$ and a single number $b$, 
\begin{equation}\label{eq:2var 2}
\begin{split}
&\trF{p+1}{q+1}{[N]}\biggl(\varray{(a,a'),\ba_1,\ldots,\ba_p\\ 
b,\bb_1,\ldots,\bb_q};z,w\biggr)\\
&=\frac{(a)_{N} (a')_{N} (b-a-a')_{N}}{(b)_{N} (N!)^2}\\
&\hspace{20pt}\times\sum_{\substack{l,m,n\ge 0\\ l+m+n=N}}
\Biggl\{\frac{(1+l)_{N-l}}{(a+l)_{N-l}}
\frac{(1+m)_{N-m}}{(a'+m)_{N-m}}
\frac{(1+n)_{N-n}}{(b-a-a'+n)_{N-n}}\\
&\hspace{140pt}
\times\trF{p}{q}{[l,m]}\biggl(\varray{\ba_1,\ldots,\ba_p\\ 
\bb_1,\ldots,\bb_q};\frac{l}{N}z,\frac{m}{N}w\biggr)\Biggr\}. 
\end{split}
\end{equation}
\item For pairs of numbers $(a,a')$ and $(b,b')$, 
\begin{equation}\label{eq:2var 3}
\begin{split}
&\trF{p+1}{q+1}{[M,N]}\biggl(\varray{(a,a'),\ba_1,\ldots,\ba_p\\ 
(b,b'),\bb_1,\ldots,\bb_q};z,w\biggr)\\
&=\frac{(a)_{M} (b-a)_{M}}{(b)_{M} M!}
\frac{(a')_N (b'-a')_N}{(b')_N N!}\\
&\hspace{20pt}\times
\sum_{\substack{0\le m\le M\\ 0\le n\le N}}
\Biggl\{\frac{(1+m)_{M-m}}{(a+m)_{M-m}}
\frac{(1+M-m)_{m}}{(b-a+M-m)_{m}}\\
&\hspace{80pt}
\times\frac{(1+n)_{N-n}}{(a'+n)_{N-n}}
\frac{(1+N-n)_{n}}{(b'-a'+N-n)_{n}}\\
&\hspace{80pt}
\times\trF{p}{q}{[m,n]}\biggl(\varray{\ba_1,\ldots,\ba_p\\ 
\bb_1,\ldots,\bb_q};\frac{m}{M}z,\frac{n}{N}w\biggr)\Biggr\}. 
\end{split}
\end{equation}
\end{enumerate}
\end{thm}
\begin{proof}
We only prove \eqref{eq:2var 2} 
(the other two can be shown similarly). 
Set $a''=b-a-a'$ and 
\[X_{m,n}=\frac{(\ba_1)_{m,n}\cdots(\ba_p)_{m,n}}
{(\bb_1)_{m,n}\cdots (\bb_q)_{m,n}m!n!}
\frac{(N+1-m-n)_{m+n}}{(Nz^{-1}-m)_m(Nw^{-1}-n)_n}. \]
Then we compute  
\begin{align*}
&\frac{(b)_N (N!)^2}{(a)_N(a')_N(a'')_N}
\trF{p+1}{q+1}{[N]}\biggl(\varray{(a,a'),\ba_1,\ldots,\ba_p\\ 
b,\bb_1,\ldots,\bb_q};z,w\biggr)\\
&=\frac{(b)_N (N!)^2}{(a)_N(a')_N(a'')_N}
\sum_{\substack{m,n\ge 0\\ m+n\le N}}
\frac{(a)_m (a')_n}{(b)_{m+n}}X_{m,n}\\
&=\sum_{\substack{m,n\ge 0\\ m+n\le N}}
\frac{(b+m+n)_{N-m-n}(N!)^2}{(a+m)_{N-m}(a'+n)_{N-n}(a'')_N}X_{m,n}\\
&=\sum_{\substack{m,n\ge 0\\ m+n\le N}}
\Biggl\{\frac{(b+m+n)_{N-m-n}((N-m-n)!)^2}
{(a+m)_{N-m-n}(a'+n)_{N-m-n}(a'')_{N-m-n}}\\
&\hspace{50pt}\times
\frac{((N-m-n+1)_{m+n})^2}{(a+N-n)_n(a'+N-m)_m(a''+N-m-n)_{m+n}}
X_{m,n}\Biggr\}.
\end{align*}
By the identity \eqref{eq:multibeta}, we have 
\begin{align*}
&\frac{(b+m+n)_{N-m-n}((N-m-n)!)^2}
{(a+m)_{N-m-n}(a'+n)_{N-m-n}(a'')_{N-m-n}}\\
&=\sum_{\substack{i,j,k\ge 0\\ i+j+k=N-m-n}}
\frac{(1+i)_{N-m-n-i}}{(a+m+i)_{N-m-n-i}}
\frac{(1+j)_{N-m-n-j}}{(a'+n+j)_{N-m-n-j}}
\frac{(1+k)_{N-m-n-k}}{(a''+k)_{N-m-n-k}}. 
\end{align*}
By substituting this into the above expression, we obtain 
\[\sum_{\substack{i,j,k,m,n\ge 0\\ i+j+k+m+n=N}}
\frac{(1+i)_{N-i}}{(a+m+i)_{N-m-i}}
\frac{(1+j)_{N-j}}{(a'+n+j)_{N-n-j}}
\frac{(1+k)_{N-m-n-k}}{(a''+k)_{N-k}} X_{m,n}. \]
If we replace $m+i$ and $n+j$ by $i$ and $j$, this is equal to 
\begin{align*}
&\sum_{\substack{i,j,k\ge 0\\ i+j+k=N}}
\Biggl\{\frac{(1+i)_{N-i}}{(a+i)_{N-i}}
\frac{(1+j)_{N-j}}{(a'+j)_{N-j}}
\frac{(1+k)_{N-k}}{(a''+k)_{N-k}} \\
&\hspace{50pt}\times\sum_{\substack{0\le m\le i\\ 0\le n\le j}}
\frac{(1+i-m)_{N-i+m}}{(1+i)_{N-i}}
\frac{(1+j-n)_{N-j+n}}{(1+j)_{N-j}}
\frac{(1+k)_{N-m-n-k}}{(1+k)_{N-k}}X_{m,n}\Biggr\} \\
&=\sum_{\substack{i,j,k\ge 0\\ i+j+k=N}}
\frac{(1+i)_{N-i}}{(a+i)_{N-i}}
\frac{(1+j)_{N-j}}{(a'+j)_{N-j}}
\frac{(1+k)_{N-k}}{(a''+k)_{N-k}} 
\sum_{\substack{0\le m\le i\\ 0\le n\le j}}
\frac{(1+i-m)_{m}(1+j-n)_{n}}{(1+N-m-n)_{m+n}}X_{m,n}. 
\end{align*}
Now, since 
\[\sum_{\substack{0\le m\le i\\ 0\le n\le j}}
\frac{(1+i-m)_{m}(1+j-n)_{n}}{(1+N-m-n)_{m+n}}X_{m,n}
=\trF{p}{q}{[i,j]}\biggl(\varray{\ba_1,\ldots,\ba_p\\ 
\bb_1,\ldots,\bb_q};\frac{i}{N}z,\frac{j}{N}w\biggr), \]
we have arrived at the formula \eqref{eq:2var 2}. 
\end{proof}

\begin{cor}
\begin{align*}
&\trF{2}{1}{[N]}\biggl(\varray{(a,a'), (b,b')\\ c};z,w\biggr)\\
&=\frac{(a)_N(a')_N(c-a-a')_N}{(c)_N(N!)^2}\\
&\hspace{20pt}\times\sum_{\substack{l,m,n\ge 0\\ l+m+n=N}}
\Biggl\{\frac{(1+l)_{N-l}}{(a+l)_{N-l}}
\frac{(1+m)_{N-m}}{(a'+m)_{N-m}}
\frac{(1+n)_{N-n}}{(c-a-a'+n)_{N-n}}\\
&\hspace{120pt}
\times\frac{(b+Nz^{-1}-l)_l}{(Nz^{-1}-l)_l}
\frac{(b'+Nw^{-1}-m)_m}{(Nw^{-1}-m)_m}\Biggr\}. 
\end{align*}
\end{cor}
\begin{proof}
This is deduced from \eqref{eq:2var 2} and 
\begin{align*}
\trF{1}{0}{[l,m]}\biggl(\varray{(b,b')\\ -};
\frac{l}{N}z,\frac{m}{N}w\biggr)
&=\trF{1}{0}{[l]}\biggl(\varray{b\\ -};\frac{l}{N}z\biggr)
\cdot\trF{1}{0}{[m]}\biggl(\varray{b'\\ -};\frac{m}{N}w\biggr)\\
&=\frac{(b+Nz^{-1}-l)_l}{(Nz^{-1}-l)_l}
\frac{(b'+Nw^{-1}-m)_m}{(Nw^{-1}-m)_m}. \qedhere
\end{align*}
\end{proof}

This Corollary is the finite analogue of the integral representation 
\begin{align*}
&F_3\biggl(\varray{a,a',b,b'\\ c};z,w\biggr)\\
&=\frac{\Gamma(c)}{\Gamma(a)\Gamma(a')\Gamma(c-a-a')}\\
&\qquad\times\int_{\substack{t,u\ge 0\\ t+u\le 1}}
\biggl\{t^{a-1}u^{a'-1}(1-t-u)^{c-a-a'-1}
(1-tz)^{-b}(1-uw)^{-b'}\biggr\}dtdu. 
\end{align*}
Theorem 6.1 does not directly yield analogous discretized integral representations 
for Appell’s functions $F_1$, $F_2$ and $F_4$, but such analogues may exist, 
possibly with different definitions of the truncations.

\section*{Acknowledgment}
The author is grateful to Prof.\ Shin-ichiro Seki for providing the information 
cited in Remark \ref{rem:zeta tilde} (2), and to Mr.\ Takumi Maesaka 
for granting permission to include it in this article. 


\end{document}